\documentclass{amsart}
\usepackage{amsmath,amssymb,amsthm,enumerate,bm}
\usepackage{csquotes}
\RequirePackage{luatex85}
\usepackage{ascmac}
\usepackage[all]{xy}

\theoremstyle{definition}
\newtheorem{dfn}{Definition}[section]
\newtheorem{thm}[dfn]{Theorem}
\newtheorem{lem}[dfn]{Lemma}
\newtheorem{cor}[dfn]{Corollary}
\newtheorem{prop}[dfn]{Proposition}
\newtheorem{conj}[dfn]{Conjecture}
\newtheorem{ex}[dfn]{Example}

\theoremstyle{remark}
\newtheorem{rem}[dfn]{Remark}

\newcommand{\lrangle}[1]{\langle #1 \rangle}
\newcommand{\Bxpm}{\mathbb B[{\bm x}^{\pm}]}
\newcommand{\Bxpmf}{\mathbb B[{\bm x}^{\pm}]_\mathrm{fcn}}
\newcommand{\Bypmf}{\mathbb B[{\bm y}^{\pm}]_\mathrm{fcn}}
\newcommand{\ZApos}{\mathbb Z^A_{\mathrm{pos}}}

\newcommand{\ZAposu}{\mathbb Z^A_{\mathrm{pos}} \cup \{ -\infty \}}
\newcommand{\ZBposu}{\mathbb Z^B_{\mathrm{pos}} \cup \{ -\infty \}}
\newcommand{\ZXposu}{\mathbb Z^{X(1)}_{\mathrm{pos}} \cup \{ -\infty \}}
\newcommand{\ZYposu}{\mathbb Z^{Y(1)}_{\mathrm{pos}} \cup \{ -\infty \}}
\newcommand{\Txf}{\mathbb T[\bm x]_{\mathrm{fcn}}}
\newcommand{\Txpmf}{\mathbb T[\bm x^{\pm}]_{\mathrm{fcn}}}
\newcommand{\Ker}{\mathrm {Ker}}
\renewcommand{\Im}{\mathrm{Im}}

\newcommand{\norm}[1]{\left\lVert #1 \right\rVert}

\newcommand{\matba}[2]{\begin{pmatrix} #1 \\ #2 \end{pmatrix}}

\newcommand{\matca}[3]{\begin{pmatrix} #1 \\ #2 \\ #3 \end{pmatrix}}
\newcommand{\matbc}[6]{\begin{pmatrix} #1 & #2 & #3 \\ #4 & #5 & #6 \end{pmatrix}}
\newcommand{\matcb}[6]{\begin{pmatrix} #1 & #2 \\ #3 & #4 \\ #5 & #6 \end{pmatrix}}
\newcommand{\matcc}[9]{\begin{pmatrix} #1 & #2 & #3 \\ #4 & #5 & #6 \\ #7 & #8 & #9 \end{pmatrix}}

\makeatletter
\@namedef{subjclassname@2020}{%
  \textup{2020} Mathematics Subject Classification}
\makeatother

\begin{document}

\title[Boolean function semirings on 1-dimensional tropical fans]{Homomorphisms between the semirings of Boolean functions on 1-dimensional tropical fans}
\author{Takaaki Ito}
\date{}
\address{Department of Mathematical sciences, Tokyo Metropolitan University, 1-1 Minami-Ohsawa, Hachioji-shi Tokyo, 192-0397, Japan}
\email{ito-t@tmu.ac.jp}

\begin{abstract}
  In \cite{ito2022local}, we stated the conjecture that any semiring homomorphism between the Boolean function semirings on 1-dimensional tropical fans has the property called \textit{geometric}.
  In this paper, we show that the conjecture is true.
  As an application, we establish a way to find all the morphisms between given 1-dimensional tropical fans.
\end{abstract}

\maketitle

\section{Introduction}
%In a classical algebraic geometry, an ideal $ito2022local \subset K[x_1, \ldots, x_n]$ gives the algebraic set $V(ito2022local) \subset \mathbb K^n$.
%However, in tropical geometry, the analogous consideration leads an unexpected result.
%Let $\mathbb T = (\mathbb R \cup {-\infty}, \oplus, \odot)$ be the tropical semifield, where
%$$a \oplus b = \max\{ a,b \}, \qquad a \odot b = a+b$$
%for any $a,b \in \mathbb T$.
%For a congruence $\mathbf E$ on the tropical polynomial semiring $\mathbb T[x_1, \ldots, x_n]$, 

In classical algebraic geometry, the correspondence between radical ideals of a polynomial rings and Zariski closed subsets in an affine space is a basic result.
Recently, the tropical version of this correspondence is studied.
Let $\mathbb T = (\mathbb R \cup \{ -\infty \}, \oplus, \odot)$ be the tropical semifield, where $\oplus$ and $\odot$ are the max operation and the standard addition respectively.
%, i.e. let $\mathbb T$ be  $\mathbb R \cup \{ -\infty \}$ as a set and endowed with the max-plus operations.
Let $\Txf = \mathbb T[x_1,\ldots, x_n]_{\mathrm{fcn}}$ (resp. $\Txpmf = \mathbb T[x_1^{\pm},\ldots, x_n^{\pm}]_{\mathrm{fcn}}$) be the semiring of functions on $\mathbb T^n$ (resp. $\mathbb R^n$) defined by polynomials (resp. Laurent polynomials) over $\mathbb T$.
For a congruence $E$ on $\Txf$ (resp. $\Txpmf$), we define the subset $\mathbf V(E)$ of $\mathbb T^n$ (resp. $\mathbb R^n$) as follows:
$$\mathbf V(E) := \{ \bm p \ | \ f(\bm p) = g(\bm p) \text{ for any } (f,g) \in E \},$$
which is called the \textit{congruence variety} associated to $E$.
Conversely, for a subset $Z$ in $\mathbb T^n$ (resp. $\mathbb R^n$), we define the congruence $\mathbf E(Z)$ on $\Txf$ (resp. $\Txpmf$) as follows:
$$\mathbf E(Z) = \{ (f,g) \ | \ f|_Z = g|_Z \}.$$
%Recently, these correspondences are studied.
%It is known that the "tropical Zariski topology" is the Euclidean topology. See \cite[Lemma 3.7.4]{GG}.
%Hence, for a subset $Z \subset \mathbb T^n$, $\mathbf V(\mathbf E(Z))$ is the closure of $Z$ with respect to the Euclidean topology, where we naturally extends the Euclidean topology on $\mathbb R^n$ to $\mathbb T^n$.
It is known that for a subset $Z \subset \mathbb T^n$, $\mathbf V(\mathbf E(Z))$ is the closure of $Z$ with respect to the Euclidean topology, see \cite[Lemma 3.7.4]{giansiracusa2022universal}.
Conversely, for a finite generated congruence $E$ on $\Txf$, the structure of $\mathbf E(\mathbf V(E))$ is studied by Bertram and Easton in \cite{bertram2017tropical}, and by Jo\'o and Mincheva in \cite{joo2018prime}.
Another related result is in \cite{song2024congruences} by Song, which studies a relation between congruences on the tropical rational function semifield and their congruence varieties.

Note that for an arbitrary congruence $E$ on $\Txf$ or $\Txpmf$, the congruence variety $\mathbf V(E)$ is not always the support of a \textit{tropical variety} (see \cite{allermann2010first} for the definition).
Thus it is natural to consider the following question: When is $\mathbf V(E)$ the support of a tropical variety?
Partial answers are given by $E = \mathbf E(|X|)$ for a tropical variety $X$, where $|X|$ is the support of $X$.
Hence it is important to study properties of $\mathbf E(|X|)$ for a tropical variety $X$.

In \cite{ito2022local}, the author studied the localized version of them, namely, studied the structure of the congruence of the form
$$\mathbf E(|X|)_{\bm p} := \left\{ (f,g) \ \middle| \ \begin{aligned} &\text{ there exists an open neighborhood $U$ of $\bm p$ } \\ &\text{ such that $f|_{U \cap |X|} = g|_{U \cap |X|}$ } \end{aligned} \right\}$$
on $\Txpmf$ for a tropical variety $X$ in $\mathbb R^n$ and a point $\bm p \in |X|$.
Then it was shown that studying such congruences is almost equivalent to studying the congruences
$$\mathbf E(|X|)_{\mathbb B} := \{ (f,g) \in \Bxpmf^2 \ | \ f|_{|X|} = g|_{|X|} \}$$
for a \textit{tropical fan} $X$ in $\mathbb R^n$, where $\mathbb B$ is the subsemifield $\{ -\infty, 0 \}$ of $\mathbb T$, which is called the \textit{Boolean semifield}, and a tropical fan is a tropical variety which is also a fan.
The semiring $\Bxpmf / \mathbf E(X)_{\mathbb B}$ is called the \textit{Boolean function semiring} on $X$.

Moreover, in \cite{ito2022local}, the case that $X$ is a 1-dimensional tropical fan is studied in more detail.
For a 1-dimensional tropical fan $X$ in $\mathbb R^n$, the author showed that $\mathbf E(|X|)_{\mathbb B}$ is the kernel of a certain semiring homomorphism $\varphi_X$ called the \textit{weighted evaluation map} of $X$.
%See Section 2 for the definition of $\varphi_X$.
%$$\varphi_X : \mathbb B[x_1^{\pm}, \ldots, x_n^{\pm}]_{\mathrm{fcn}} \to \ZXposu$$
%called the \textit{weighted evaluation map} of $X$, where the definition of $\ZXposu$ will be in subsection \ref{pre B and Z}.
%The kernel of $\varphi_X$ is $\mathbf E(|X|)_{\mathbb B}$.
It was also shown that the correspondence $X \mapsto \varphi_X$ gives a faithful contravariant functor, and the author stated the conjecture that the functor is full.

The purpose of this paper is to show that the conjecture is true.
As we will see in Section 2, it is sufficient to show the following, which is our main theorem.

\begin{thm}
  Let $A,B$ be finite sets, $R$ a Laurent-generated subsemiring of $\ZAposu$, and $\nu : R \to \ZBposu$ a semiring homomorphism.
  Then $\nu$ is geometric.
\end{thm}
See Section 2 for the definitions of Laurent-generated subsemirings, $\ZApos$, and geometric homomorphisms.
In the proof of the main theorem, we also obtain the following result.

\begin{thm}
  Let $Z$ be a nonempty closed subset of $\mathbb R^n$ which satisfies $\mathbb R_{\geq 0} Z = Z$, where $\mathbb R_{\geq 0} Z = \{ t \bm p \in \mathbb R^n \ | \ t \in \mathbb R_{\geq 0}, \bm p \in Z \}$.
  Then $\mathbf V(\mathbf E(Z)_{\mathbb B}) = Z$.
\end{thm}

Moreover, as an application of the main theorem, we establish a way to find all the semiring homomorphism between the Boolean function semirings of given 1-dimensional tropical fans.
This means that we can also find all the morphisms between given 1-dimensional tropical fans.

This paper is organized as follows.
In Section 2, we see some basic notions of semirings, some basic properties of the Boolean Laurent polynomial function semiring, the definition of the semiring $\ZAposu$, and some properties of them.
Moreover, we briefly review the result of \cite{ito2022local} and see why our main theorem implies that the functor $X \mapsto \varphi_X$ is full.
In Section 3, we show our main theorem.
In Section 4, we give a way to find all the semiring homomorphism between the Boolean function semirings of given 1-dimensional tropical fans.
We see some examples.
\section{Preliminary}

\subsection{Semirings}

%In this paper, semirings are always commutative and has the multiplicative identity element.
%Thus, a \textit{semiring} is a set $R$ endowed with two binary operations $+$ and $\cdot$ and two fixed elements $0$ and $1$ such that
A \textit{semiring} is a set $R$ endowed with two binary operations $+$ and $\cdot$ and two fixed elements $0$ and $1$ such that
\begin{enumerate}[(1)]
\item $(R, +)$ is a commutative monoid with identity element 0,
\item $(R, \cdot)$ is a commutative monoid with identity element 1,
\item $a(b+c) = ab + ac$ for any $a,b, c \in R$, and
\item $0 \cdot a = 0$ for any $a \in R$,
\end{enumerate}
where $ab$ means $a \cdot b$.
An element $a$ of a semiring $R$ is \textit{invertible} if there exists $b \in R$ such that $ab=1$.
The set of invertible elements of $R$ is denoted by $R^{\times}$, which forms an abelian group with respect to the multiplication.
We call $R^{\times}$ the \textit{unit group} of $R$.
A semiring $R$ is called a \textit{semifield} if $0 \neq 1$ and $R^{\times} = R \setminus \{ 0 \}$.

Let $R$ be a semiring.
A subset $E \subset R \times R$ is a \textit{congruence} on $R$ if 
\begin{enumerate}[(1)]
  \item $(a,a) \in E$ for any $a \in R$,
  \item if $(a,b) \in E$, $(b,a) \in E$,
  \item if $(a,b),(b,c) \in E$, $(a,c) \in E$,
  \item if $(a,b),(c,d) \in E$, $(a+c,b+d) \in E$, and
  \item if $(a,b),(c,d) \in E$, $(ac,bd) \in E$.
\end{enumerate} 

Any congruence $E$ on $R$ is an equivalence relation on $R$.
Moreover, the quotient set $R/E$ naturally has a semiring structure with respect to the operations
$$[a]+[b] = [a+b], \qquad [a]\cdot[b] = [ab],$$
where $[a]$ means the equivalence class of $a$.
The semiring $R/E$ is called the \textit{quotient semiring} of $R$ by $E$.

A subset $R'$ of a semiring $R$ is a \textit{subsemiring} of $R$ if $R'$ is itself a semiring with respect to the operations same to $R$.
For a subset $S \subset R$, the subsemiring of $R$ \textit{generated by} $S$ is the smallest subsemiring of $R$ containing $S$, which is denoted by $\lrangle{S}$.
Explicitly, $\lrangle{S}$ consists of the elements of the form
$$a_{11}a_{12} \cdots a_{1n_1} + \cdots + a_{m1}a_{m2} \cdots a_{mn_m},$$
where each $a_{ij}$ is in $S$.

Let $R_1, R_2$ be semirings.
A map $\varphi : R_1 \to R_2$ is a \textit{semiring homomorphism} if
\begin{enumerate}[(1)]
\item $\varphi (a+b) = \varphi (a)+\varphi (b)$ for any $a,b \in R_1$,
\item $\varphi(ab) = \varphi(a)\varphi(b)$ for any $a,b \in R_1$, 
\item $\varphi(0)=0$, and $\varphi(1) = 1$. 
\end{enumerate}

A bijective semiring homomorphism is called a \textit{semiring isomorphism}.
Two semirings $R_1, R_2$ are \textit{isomorphic} if there exists a semiring isomorphism from $R_1$ to $R_2$.

For a semiring homomorphism $\varphi : R_1 \to R_2$ , the \textit{kernel} of $\varphi$ is defined as
$$\Ker (\varphi) := \{ (f,g) \in R_1 \times R_1 \ | \ \varphi(f)=\varphi(g) \},$$
which is a congruence on $R_1$.
The \textit{image} of $\varphi$ is defined as
$$\Im (\varphi) := \{ \varphi(f) \in R_2 \ | \ f \in R_1 \},$$
which is a subsemiring of $R_2$.
$R_1/\Ker (\varphi)$ and $\Im (\varphi)$ are naturally isomorphic.

Let $R$ be a semiring.
An $R$\textit{-algebra} is a semiring $S$ endowed with a semiring homomorphism $R \to S$.
For two $R$-algebras $S,T$, a semiring homomorphism $\varphi : S \to T$ is a \textit{homomorphism of $R$-algebra} if the following diagram is commutative:
$$\xymatrix{R \ar[d] \ar[rd] & \\ S \ar[r]_{\varphi} & T}$$

%Let $\{ I_{\alpha} \}_{\alpha \in A}$ be a family of congruences of a semiring $R$.
%Then the intersection $\cap_{\alpha \in A} I_{\alpha}$ is also a congruence on $R$.

%For any subset $\Lambda$ of $R \times R$, let $\langle \Lambda \rangle$ be the intersection of all congruences on $R$ containing $\Lambda$.
%Then $\langle \Lambda \rangle$ is the smallest congruence on $R$ containing $\Lambda$, which is called the congruence \textit{generated by} $\Lambda$.

Let $n$ be a positive integer.
%In the rest of this section, we fix a positive integer $n$.
For a semiring $R$, the \textit{Laurent polynomial semiring} $R[x_1^{\pm}, \ldots, x_n^{\pm}]$ over $R$ is naturally defined, which is an $R$-algebra.
For an $R$-algebra $S$ and any invertible elements $a_1, \ldots, a_n \in S^{\times}$, the map
$$\varphi : R[x_1^{\pm}, \ldots, x_n^{\pm}] \to S, \qquad \varphi(f) = f(a_1, \ldots, a_n)$$
is a homomorphism of $R$-algebra.

A semiring $R$ is (\textit{additively}) \textit{idempotent} if $a+a=a$ for any $a \in R$.
%In this paper, “idempotent” always means additively idempotent.
%If $R$ is an idempotent semiring, $R$ has the following order.
%$$a \leq b \Longleftrightarrow a+b = b.$$
%In this paper, the order on an idempotent semiring is always this one unless otherwise mentioned

\subsection{Boolean Laurent polynomial function semirings}
%\subsection{Boolean Laurent polynomial semirings and the semiring $\ZAposu$}
\label{pre B and Z}

Let $\mathbb B := \{ 0, -\infty\}$.
Then $\mathbb B$ forms an idempotent semifield with respect to the addition
$$a \oplus b = \max\{a,b\},$$
and the multiplication
$$a \odot b = a+b.$$
Thus $\mathbb B$ is called the \textit{Boolean semifield}.
The semiring of the form $\mathbb B[{x_1}^{\pm}, \ldots, x_n^{\pm}]$ for some $n$ is called the \textit{Boolean Laurent polynomial semiring}.
For an integer vector $\bm u = (u_1, \ldots, u_n) \in \mathbb Z^n$, we denote by $\bm x^{\bm u} = x_1^{u_1} \cdots x_n^{u_n}$.
A Boolean Laurent polynomial can be written as
$$\bigoplus_{i=1}^{m} \bm x^{\bm u_i}$$
for some $\bm u_1, \ldots, \bm u_m \in \mathbb Z^n$.
Note that the omitted coefficients are 0 since $0$ is the multiplicative identity element of $\mathbb B$.

In the rest of this section, we fix a positive integer $n$ and denote $\mathbb B[{\bm x}^{\pm}] = \mathbb B[{x_1}^{\pm}, \ldots, x_n^{\pm}]$. 
%We fix a positive integer $n$ in this section and call the semiring $\mathbb B[{\bm x}^{\pm}] = \mathbb B[{x_1}^{\pm}, \ldots, x_n^{\pm}]$ the \textit{Boolean Laurent polynomial semiring}.

It is easily shown that a semiring $R$ is idempotent if and only if $R$ is a $\mathbb B$-algebra.
Hence for any $P \in \Bxpm$ and invertible elements $a_1, \ldots, a_n$ of an idempotent semiring $R$, the value $P(a_1, \ldots, a_n)$ is well-defined.
%Let $R$ be an idempotent semiring and $a_1, \ldots, a_n \in R^{\times}$ any invertible elements.
%Since a Boolean Laurent polynomial $P \in \Bxpm$ is constructed by finitely many addition and multiplication from $x_1^{\pm}, \ldots, x_n^{\pm}$, $P(a_1, \ldots, a_n)$ is naturally defined.
%In other words, we may \textit{substitute} any invertible elements $a_1, \ldots, a_n \in R^{\times}$ to $x_1, \ldots, x_n$ in $P$.

Each element of $\Bxpm$ defines a function on $\mathbb R^n$.
Explicitly, the polynomial $\bigoplus_{i=1}^{m} \bm x^{\bm u_i}$ defines the function
$$\mathbb R^n \ni \bm x \mapsto \underset{i}{\max} \{ \bm u_i \cdot \bm x \} \in \mathbb R,$$
where $\cdot$ is the standard inner product.
We denote $\overline P$ the function on $\mathbb R^n$ defined by $P \in \Bxpm$.
%For example, $\overline {x_i}(\bm p)$ means the $i$-th entry of $\bm p$.
Consider the following congruence:
$$\left\{ (P,Q) \in \Bxpm^2 \ | \ \overline P = \overline Q \right\}.$$
We denote $\Bxpmf$ the quotient semiring of $\Bxpm$ by the above congruence.
We call $\Bxpmf$ the \textit{Boolean Laurent polynomial function semiring}.
We may regard each element of $\Bxpmf$ as a function on $\mathbb R^n$.

%\begin{rem}
%  Some authors use different notations.
%  In \cite{song2024congruences}, 
%\end{rem}

The following are basic properties of Boolean Laurent polynomial functions.

%\begin{lem}[\cite{ito2022local}, Lemma 2.4]
%  \label{Pt tP}
%  For any polynomial $P \in \Bxpm$, vector $\bm p \in \mathbb R^n$ and $t \geq 0$, $P(t \bm p) = t P(\bm p)$, where the multiplication is the standard one.
%\end{lem}

\begin{lem}[{\cite[Lemma 2.4]{ito2022local}}]
  \label{ft tf}
  For any function $f \in \Bxpmf$, vector $\bm p \in \mathbb R^n$ and $t \geq 0$,
  $$f(t \bm p) = t f(\bm p),$$
  where the multiplication is the standard one.
\end{lem}

%A \textit{ray} spanned by a vector $\bm p \in \mathbb R^n \setminus \{ \mathbf 0 \}$ is the set $\{ t\bm p \in \mathbb R^n \ | \ t \geq 0 \}$.

\begin{lem}
  \label{point ray}
  Let $\bm p \in \mathbb R^n \setminus \{ \mathbf 0 \}$ be any nonzero vector and let
  $$\rho = \{ t \bm p \in \mathbb R^n \ | \ t \leq 0 \}.$$
  Then, for any $f, g \in \Bxpmf$, $f|_{\rho} = g|_{\rho}$ if and only if $f(\bm p) = g(\bm p)$.
\end{lem}

\begin{proof}
  Clear by Lemma \ref{ft tf}.
\end{proof}

Let $Z \subset \mathbb R^n$ be any subset.
We define the congruence $\mathbf E(Z)_{\mathbb B}$ on $\Bxpmf$ as
$$\mathbf E(Z)_{\mathbb B} := \{ (f,g) \ | \ f|_Z = g|_Z \}.$$

\begin{lem}
  \label{nonneg hull}
  For any subset $Z \subset \mathbb R^n$,
  $$\mathbf E(Z)_{\mathbb B} = \mathbf E(\mathbb R_{\geq 0} Z)_{\mathbb B},$$
  where $\mathbb R_{\geq 0} Z := \{ t\bm p \ | \ t \geq 0 \text{ and } \bm p \in Z \}$. 
\end{lem}

\begin{proof}
  Clear by Lemma \ref{point ray}.
\end{proof}

Conversely, for a congruence $E$ on $\Bxpmf$, we define the \textit{congruence variety} $\mathbf V(E)$ as
$$\mathbf V(E) = \{ \bm p \in \mathbb R^n \ | \ f(\bm p) = g(\bm p) \text{ for any } (f,g) \in E \}.$$

%\begin{lem}
%  \label{polyn evalu}
%  Let $R$ be an idempotent semiring and take any invertible elements $a_1, \ldots, a_n$ of $R$.
%  Then there exists a unique semiring homomorphism $\varphi : \Bxpm \to R$ such that $\varphi(x_i) = a_i$ for any $i$.
%\end{lem}

%\begin{lem}
%  \label{polyn evalu}
%  Let $R$ be an idempotent semiring and take any invertible elements $a_1, \ldots, a_n$ of $R$.
%  The map $\varphi : \Bxpm \to R$, \ $\varphi(P) = P(a_1, \ldots, a_n)$ is a semiring homomorphism.
%\end{lem}

%\begin{proof}
  %%%%%
%\end{proof}

\subsection{The semiring $\ZAposu$}

In this subsection, we fix a finite set $A$.
%Let $A$ be a finite set.
We denote $(\mathbb Z \cup \{ -\infty \})^A$ the set of maps from $A$ to $\mathbb Z \cup \{ -\infty \}$.
This set forms an idempotent semiring with respect to the addition
$$(F \oplus G)(a) = \max\{F(a), G(a)\} \text{ for any } a \in A$$
and the multiplication
$$(F \odot G)(a) = F(a)+G(a) \text{ for any } a \in A.$$

\begin{rem}
%  When we fix a bijection $A \to \{ 1, \ldots, r \}$, we may identify $(\mathbb Z \cup \{ -\infty \})^A$ with $(\mathbb Z \cup \{ -\infty \})^r$.
  In practice, the elements of $A$ are often given indices such as $A = \{ a_1, \ldots, a_r \}$.
  In that case, we may identify $(\mathbb Z \cup \{ -\infty \})^A$ with $(\mathbb Z \cup \{ -\infty \})^r$.
  Hence an element of $(\mathbb Z \cup \{ -\infty \})^A$ can be regarded as an $r$-dimensional vector.
  We will use this identification in the examples in Section 4.
\end{rem}

For an element $F \in (\mathbb Z \cup \{ -\infty \})^A$, we define the \textit{degree} $\deg F$ of $F$ as
$$\deg F = \bigodot_{a \in A} F(a) = \sum_{a \in A} F(a).$$
Let
$$\mathbb Z^A_{\mathrm{pos}} := \{ F \in (\mathbb Z \cup \{ -\infty \})^A \ | \ \deg F \geq 0 \},$$
and
$$\mathbb Z^A_0 := \{ F \in (\mathbb Z \cup \{ -\infty \})^A \ | \ \deg F = 0 \}.$$
Then $\ZAposu$ is a subsemiring of $(\mathbb Z \cup \{ -\infty \})^A$, where, by abuse of notation, we denote by $-\infty$ the map $a \mapsto -\infty$ for any $a \in A$.
Also $\mathbb Z^A_0$ is an abelian group with respect to the multiplication, which is a free abelian group of rank $|A|-1$.

%These semirings appear in the local theory of functions on tropical varieties. See \cite{ito2022local}.

\begin{lem}[{\cite[the argument before Proposition 5.4]{ito2022local}}]
  We have $(\ZAposu)^{\times} = \mathbb Z^A_0$.
\end{lem}

\begin{lem}
  \label{gen by 0}
  If $|A| \geq 2$, $\ZAposu$ is generated by $\mathbb Z^A_0$ as a semiring.
\end{lem}

\begin{proof}
  Fix distinct elements $a_1, a_2 \in A$.
  For any $F \in \ZAposu$, consider the following two functions:
  $$F_1(a) = \begin{cases}
    F(a) - \deg F & a = a_1, \\
    F(a) & \text{otherwise},
  \end{cases} \qquad 
  F_2(a) = \begin{cases}
    F(a) - \deg F & a = a_2, \\
    F(a) & \text{otherwise}.
  \end{cases}$$
  Then $F_1, F_2 \in \mathbb Z^A_0$ and $F_1 \oplus F_2 = F$.
\end{proof}

\begin{lem}[{\cite[Lemma 5.20]{ito2022local}}]
  \label{unit group}
  Let $F_1, \ldots, F_n$ be any elements in $\mathbb Z^A_0$.
  Then the semiring homomorphism $\Bxpm \to \ZAposu, \ P \mapsto P(F_1, \ldots, F_n)$ induces a semiring homomorphism $\Bxpmf \to \ZAposu$.
  %  Then there exists a unique semiring homomorphism $\varphi:\Bxpmf \to \ZAposu$ such that $\varphi(\overline {x_i}) = F_i$ for any $i$.
\end{lem}

This lemma means that, if $f = \overline P = \overline Q$ for $P,Q \in \Bxpm$, then $P(F_1, \ldots, F_n) = Q(F_1, \ldots, F_n)$.
Thus we write that value as $f(F_1, \ldots, F_n)$.
In other words, we may \textit{substitute} any invertible elements $F_1, \ldots, F_n \in \ZAposu$ to $\overline{x}_1, \ldots, \overline{x}_n$ in $f$.

\begin{lem}
  \label{elementwise}
  Let $F_1, \ldots, F_n \in \mathbb Z^A_0$ be any invertible elements, and let $f \in \Bxpmf$ be any Boolean Laurent polynomial function.
  Then, for any $a \in A$,
  $$f(F_1, \ldots, F_n)(a) = f(F_1(a), \ldots, F_n(a)).$$
\end{lem}

\begin{proof}
  In general, for any $F, G \in \mathbb Z^A_0$ and $a \in A$,
  $$(F \oplus G)(a) = F(a) \oplus G(a),$$
  $$(F \odot G)(a) = F(a) \odot G(a),$$
  and
  $$(F^{-1})(a) = F(a)^{-1}$$
  by the definitions of operations on $\ZAposu$.
  Since $f$ is constructed by finitely many addition and multiplication from $\overline{x_1}^{\pm 1}, \ldots, \overline{x_n}^{\pm 1}$, the statement holds.
\end{proof}

\subsection{One-dimensional fans and their Boolean function semirings}

%In the proof of the main theorem, we use some 1-dimensional fans in $\mathbb R^n$.
%Thus we recall the definition.

A \textit{ray} in $\mathbb R^n$ is the subset of $\mathbb R^n$ of the form $\{ t\bm d \ | \ t \geq 0 \}$ for some vector $\bm d \in \mathbb R^n \setminus \{ \mathbf 0 \}$.
For a vector $\bm d \in \mathbb R^n \setminus \{ \mathbf 0 \}$, the ray $\{ t\bm d \ | \ t \geq 0 \}$ is called the ray \textit{spanned by} $\bm d$.
A \textit{direction vector} of a ray $\rho$ is the vector which spans $\rho$.
A ray is \textit{rational} if it is spanned by an integer vector.
A \textit{1-dimensional fan} in $\mathbb R^n$ is a set of the form $\{ \{ \mathbf 0 \}, \rho_1, \rho_2, \ldots, \rho_r \}$ for some positive integer $r$ and rational rays $\rho_1, \ldots, \rho_r$ in $\mathbb R^n$.
We denote $X(1)$ the set of rays in a 1-dimensional fan $X$.
The \textit{support} $|X|$ of a 1-dimensional fan $X$ is the union of all the rays in $X$.

For a 1-dimensional fan $X$ in $\mathbb R^n$, we denote $\mathbf E(X)_{\mathbb B} := \mathbf E(|X|)_{\mathbb B}$.
That is,
$$\mathbf E(X)_{\mathbb B} := \{ (f,g) \in \Bxpmf^2 \ | \ f|_{|X|} = g|_{|X|} \}.$$
We call the quotient semiring $\Bxpmf / \mathbf E(X)_{\mathbb B}$ the \textit{Boolean function semiring} on $X$.

\subsection{Boolean function semirings on 1-dimensional tropical fans}
\label{review ito2022local}

We briefly review the result of \cite{ito2022local}.
A \textit{1-dimensional tropical fan} in $\mathbb R^n$ is a weighted 1-dimensional fan in $\mathbb R^n$ satisfying the \textit{balancing condition} (see \cite{ito2022local} for detail).
Let $X$ be a 1-dimensional tropical fan in $\mathbb R^n$.
For each ray $\rho \in X(1)$, let $\bm d_{\rho}$ be the primitive direction vector of $\rho$, and let $w_{\rho}$ be the weight of $\rho$.
Then we define the \textit{weighted evaluation map} $\varphi_X$ of $X$ as
$$\varphi_X: \Bxpmf \to \mathbb Z^{X(1)}_{\mathrm{pos}} \cup \{ -\infty \}, \qquad \varphi_X(f)(\rho) = w_{\rho} f(\bm d_{\rho}).$$
The image is in fact contained in $\mathbb Z^{X(1)}_{\mathrm{pos}} \cup \{ -\infty \}$, see \cite[Proposition 5.4]{ito2022local}.
The kernel $\Ker(\varphi_X)$ coincides with $\mathbf E(X)_{\mathbb B}$, and hence the Boolean function semiring $\Bxpmf / \mathbf E(X)_{\mathbb B}$ on $X$ is isomorphic to a subsemiring of $\mathbb Z^{X(1)}_{\mathrm{pos}} \cup \{ -\infty \}$.

The correspondence $X \mapsto \varphi_X$ defines a contravariant functor.
More precisely, it is a functor between the following two categories.
The first one is the category whose objects are 1-dimensional tropical fans and whose morphisms are defined as follows.
%, where the morphisms are defined as follows.
Let $X,Y$ be 1-dimensional tropical fans in $\mathbb R^n, \mathbb R^m$ respectively.
A morphism $\mu : X \to Y$ is a map from $|X|$ to $|Y|$ which is the restriction of a linear map from $\mathbb R^n$ to $\mathbb R^m$ defined by an integer matrix.
The second one is the category whose objects are semiring homomorphisms from $\mathbb B[x_1^{\pm},\ldots, x_n^{\pm}]_{\mathrm{fcn}}$ to $\ZAposu$ for some positive integer $n$ and a finite set $A$, and whose morphisms are defined as follows.
Let $A,B$ be finite sets.
%Let $\varphi : \Bxpmf \to \ZAposu$, $\psi : \Bypmf \to \ZBposu$ be semiring homomorphisms, where $\Bypmf = \mathbb B [y_1^{\pm},\ldots, y_m^{\pm}]_{\mathrm{fcn}}$.
Let $\varphi : \Bxpmf \to \ZAposu$, $\psi : \Bypmf \to \ZBposu$ be semiring homomorphisms.
A morphism $\nu : \varphi \to \psi$ is a semiring homomorphism from $\Im(\varphi)$ to $\Im(\psi)$.
For a morphism $\mu : X \to Y$ between 1-dimensional tropical fans, the corresponding morphism is defined as follows:
The morphism $\mu$ induces the semiring homomorphism
$$\mu^{*} : \Bxpmf / \mathbf E(Y)_{\mathbb B} \to \Bypmf / \mathbf E(X)_{\mathbb B}, \qquad f \mapsto f \circ \mu.$$
Then the corresponding morphism is the following composition:
$$\Im(\varphi_Y) \xrightarrow{\cong} \Bxpmf / \mathbf E(Y)_{\mathbb B} \xrightarrow{\mu^*} \Bypmf / \mathbf E(X)_{\mathbb B} \xrightarrow{\cong} \Im(\varphi_X).$$

Let $\Phi$ be the above functor.
In \cite{ito2022local}, we showed that $\Phi$ is faithful.
We also conjectured that $\Phi$ is full.
%In other words, let $X, Y$ be 1-dimensional tropical fans in $\mathbb R^n, \mathbb R^m$ respectively.
%Let $\nu : \varphi_Y$ \to $\varphi_X$ be a morphism.
%Then it is conjectured that there exists a morphism $\mu : X \to Y$ such that $\Phi(\mu) = \nu$.
%We defined \textit{geometric homomorphisms} in \cite{ito2022local}.
%See the next subsection for the definition.
The following was shown.

\begin{prop}[{\cite[Proposition 6.6 \& Proposition 6.8]{ito2022local}}]
  \label{full geom}
  Let $X, Y$ be 1-dimensional tropical fans in $\mathbb R^n, \mathbb R^m$ respectively.
  Let $\nu : \varphi_Y \to \varphi_X$ be a morphism.
  Then the following are equivalent:
  \begin{enumerate}[(1)]
    \item there exists a morphism $\mu : X \to Y$ such that $\Phi(\mu) = \nu$, and
    \item $\nu$ is \textit{geometric}.
  \end{enumerate}
\end{prop}

The definition of geometric homomorphism is in the next subsection.
Thus, to show that $\Phi$ is full, it is sufficient to show that the following conjecture is true.

\begin{conj}
  \label{conj0}
  Let $X, Y$ be 1-dimensional tropical fans in $\mathbb R^n, \mathbb R^m$ respectively.
  Then any morphism $\nu : \varphi_Y \to \varphi_X$ is geometric.
\end{conj}

\subsection{Geometric homomorphisms}

We now recall the definition of geometric homomorphisms.
To shorten the description, we introduce some new terms.
\begin{dfn}
  Let $R$ be a semiring and let $S$ be a subset of $R^{\times}$.
  Let $S^{-1} := \{ s^{-1} \in R \ | \ s \in S \}$.
  Then the subsemiring $R' := \lrangle{S \cup S^{-1}}$ is called the subsemiring of $R$ \textit{Laurent-generated by} $S$.
  The set $S$ is called a \textit{Laurent-generating set} of $R'$.
  If a subsemiring $R'$ of $R$ is Laurent-generated by some set, then we say that $R'$ is a \textit{Laurent-generated subsemiring}.
  Moreover, if $R'$ is Laurent-generated by some finite set, then we say that $R'$ is a \textit{finitely Laurent-generated subsemiring}.
\end{dfn}

Explicitly, the subsemiring Laurent-generated by $S$ consists of the elements of the form
$$a_{11}a_{12} \cdots a_{1n_1} + \cdots + a_{m1}a_{m2} \cdots a_{mn_m},$$
where each $a_{ij}$ is in $S \cup S^{-1}$.

Note that, since $(\ZAposu)^{\times} = \mathbb Z^A_0$ is a free abelian group of finite rank, any Laurent-generated subsemiring of $\ZAposu$ is finitely Laurent-generated.

\begin{lem}[{\cite[Lemma 5.8 and Lemma 5.18]{ito2022local}}]
  Let $A$ be a finite set, let $R$ be a Laurent-generated subsemiring of $\ZAposu$, and let $\{ F_1 ,\ldots, F_n \}$ be a Laurent-generating set of $R$.
  Then
  $$R^{\times} = R \cap \mathbb Z^A_0 = \lrangle{F_1, \ldots, F_n}_0,$$
  where $\lrangle{F_1, \ldots, F_n}_0$ is the subgroup of $\mathbb Z^A_0$ generated by $\{ F_1 ,\ldots, F_n \}$.
\end{lem}

\begin{dfn}
  Let $A,B$ be finite sets, let $R, R'$ be Laurent-generated subsemirings of $\ZAposu$ and $\ZBposu$ respectively, and let $S \subset R \cap \mathbb Z^A_0$ be a finite subset.
  A semiring homomorphism $\nu : R \to R'$ is \textit{geometric with respect to} $S$ if for any $b \in B$, there exists $a \in A$ and a nonnegative rational number $t$ such that $\nu(F)(b) = t F(a)$ for any $F \in S$.
\end{dfn}

In the situation of the above definition, it is clear by definition that $\nu : R \to R'$ is geometric with respect to $S$ if and only if the composition $R \xrightarrow{\nu} R' \hookrightarrow \ZBposu$ is so.

%Note that we slightly change the definition of geometricity from \cite{ito2022local}.
%That is, we allow $S$ to be an infinite set, which is for simplifying the descriptions.
%The following proposition, which seems to same as \cite[Corollary 6.15]{ito2022local}, also holds for this definition.

\begin{prop}[{\cite[Corollary 6.15]{ito2022local}}]
  Let $A,B$ be finite sets, let $R, R'$ be Laurent-generated subsemirings of $\ZAposu$ and $\ZBposu$ respectively, and let $\nu : R \to R'$ be a semiring homomorphism.
  Then the following are equivalent:
  \begin{enumerate}[(1)]
    \item $\nu$ is geometric with respect to a Laurent-generating set of $R$.
    \item $\nu$ is geometric with respect to any finite subset of $R \cap \mathbb Z^A_0$.
  \end{enumerate}
\end{prop}

%\begin{proof}
%  
%\end{proof}

%Thus the following definition is well-defined and same to \cite{ito2022local}.
Thus the following terminology makes sense.

\begin{dfn}
  Let $A,B$ be finite sets and let $R, R'$ be Laurent-generated subsemirings of $\ZAposu$ and $\ZBposu$ respectively.
  Then a semiring homomorphism $\nu : R \to R'$ is \textit{geometric} if it is geometric with respect to some Laurent-generating set of $R$.
\end{dfn}

%In \cite{ito2022local}, we gave the following conjecture:

\begin{conj}
  \label{conj}
  Let $A,B$ be finite sets, and let $R$ be a Laurent-generated subsemiring of $\ZAposu$.
  Then \textit{any} semiring homomorphism $\nu : R \to \ZBposu$ is geometric.
\end{conj}

The purpose of this paper is to prove that this conjecture is true.

Note that if Conjecture \ref{conj} is true, then Conjecture \ref{conj0} is also true.
Indeed, let $X, Y$ be 1-dimensional tropical fans in $\mathbb R^n, \mathbb R^m$ respectively, and let $\varphi_Y : \Bypmf \to \ZYposu$ be the weighted evaluation map of $Y$.
Then $\Im(\varphi_Y)$ is Laurent-generated by $\{ \overline{y_1}, \ldots, \overline{y_m} \}$.
If Conjecture \ref{conj} is true, then any semiring homomorphism $\nu : \Im(\varphi_Y) \to \mathbb Z^{X(1)}_{\mathrm{pos}} \cup \{ -\infty \}$ is geometric.
This means that any semiring homomorphism $\nu : \Im(\varphi_Y) \to \Im(\varphi_X)$ is also geometric.

%Later we use the following lemma.

%\begin{lem}
%  \label{L-gen image}
%  Let $R$ be an idempotent semiring, $S=\{a_1 \ldots, a_n\}$ a subset of $R^{\times}$, and $R'$ the subsemiring Laurent-generated by $S$.
%  Then $R'$ is the image of the semiring homomorphism $\varphi : \mathbb B[x_1^{\pm}, \ldots, x_n^{\pm}] \to R$ defined as $\varphi(P) = P(a_1,\ldots, a_n)$.
%\end{lem}

%\begin{proof}
%  Note that $\varphi$ is a semiring homomorphism by Lemma \ref{polyn evalu}.
%  Then this lemma follows from the explicit description of the elements of a Laurent-generated subsemiring.
%\end{proof}

\section{Proof of the main theorem}

In this section, we fix finite sets $A, B$, a Laurent-generated subsemiring $R$ of $\ZAposu$, a Laurent-generating set $\{ F_1, \ldots, F_n \}$ of $R$, and a semiring homomorphism $\nu : R \to \ZBposu$.
We write as $\Bxpmf = \mathbb B[x_1^{\pm}, \ldots, x_n^{\pm}]_{\mathrm{fcn}}$ and $\bm F=(F_1, \ldots, F_n)$.

\subsection{First key lemma}

We define the (at most) 1-dimensional fan $X_{\bm F}$ in $\mathbb R^n$ as follows:
For each $a \in A$, consider the vector
$$\bm d_a := \begin{pmatrix}
    F_1(a) \\ \vdots \\ F_n(a)
\end{pmatrix},$$
and let $\rho_a := \{ t \bm d_a \in \mathbb R^n \ | \ t \geq 0 \}$.
Note that $\rho_a$ is possibly $\{ \mathbf 0 \}$.
%where if that vector is $\mathbf 0$, we define $\rho_a = \{ \mathbf 0 \}$.
Then let $X_{\bm F} := \{ \{ \mathbf 0 \} \} \cup \{ \rho_a \ | \ a \in A \}$.

The following is the first key lemma in the proof of the mein theorem.

\begin{lem}
  \label{key 1}
  %Let $R$ be a Laurent-generated subsemiring of $\ZAposu$, let $\{ F_1, \ldots, F_n \}$ be its Laurent-generating set, and let $\bm F = (F_1, \ldots, F_n)$.
  Let $\varphi : \Bxpmf \to \ZAposu$ be the semiring homomorphism defined as $\varphi(f) = f(F_1, \ldots, F_n)$.
  Then $\Im(\varphi) = R$ and $\Ker(\varphi) = \mathbf E(X_{\bm F})_{\mathbb B}$.
\end{lem}

\begin{proof}
  The equality $\Im(\varphi) = R$ is clear.
  For any pair $(f,g) \in \Bxpmf$, $(f,g) \in \Ker(\varphi)$ if and only if $f(F_1, \ldots, F_n) = g(F_1, \ldots, F_n)$, which means that
  $$f(F_1, \ldots, F_n)(a) = g(F_1, \ldots, F_n)(a)$$
  for any $a \in A$.
  By Lemma \ref{elementwise}, it is equivalent to
  $$f(F_1(a), \ldots, F_n(a)) = g(F_1(a), \ldots, F_n(a))$$
  for any $a \in A$.
  This equality also can be written as $f(\bm d_a) = g(\bm d_a)$.
  Let
  $$Z := \left\{ \bm d_a \in \mathbb R^n \ \middle| \ a \in A \right\}.$$
  Then $(f,g) \in \Ker(\varphi)$ if and only if
  $(f,g) \in \mathbf E(Z)_{\mathbb B},$ hence $\Ker(\varphi) = \mathbf E(Z)_{\mathbb B}$.
  By Lemma \ref{nonneg hull}, $\mathbf E(Z)_{\mathbb B} = \mathbf E(\mathbb R_{\geq 0} Z)_{\mathbb B}$, but $\mathbb R_{\geq 0} Z = |X_{\bm F}|$ by the definition of $X_{\bm F}$.
  Therefore $\mathbf E(\mathbb R_{\geq 0} Z)_{\mathbb B} = \mathbf E(|X_{\bm F}|)_{\mathbb B} = \mathbf E(X_{\bm F})_{\mathbb B}$.
\end{proof}

\begin{cor}
  For any $b \in B$, 
  $$\begin{pmatrix} \nu(F_1)(b) \\ \vdots \\ \nu(F_n)(b) \end{pmatrix} \in \mathbf V(\mathbf E(X_{\bm F})_{\mathbb B}).$$
\end{cor}

\begin{proof}
  Let $\varphi$ be the semiring homomorphism in Lemma \ref{key 1}.
  Then $\Im(\varphi) = R$ and $\Ker(\varphi) = \mathbf E(X_{\bm F})_{\mathbb B}$.
  Thus $\varphi$ induces an isomorphism
  $$\overline{\varphi} : \Bxpmf / \mathbf E(X_{\bm F})_{\mathbb B} \to R.$$
  There is a composition homomorphism
  $$\nu \circ \overline{\varphi} : \Bxpmf / \mathbf E(X_{\bm F})_{\mathbb B} \to \ZBposu.$$
  Since this homomorphism is well-defined, if $(f,g) \in \mathbf E(X_{\bm F})_{\mathbb B}$, then $(\nu \circ \varphi)(f) = (\nu \circ \varphi)(g)$.
  The left hand side is
  $$(\nu \circ \varphi)(f) = \nu(f(F_1, \ldots, F_n)) = f(\nu(F_1), \ldots, \nu(F_n)),$$
  where the last equality holds because $\nu$ is a semiring homomorphism.
  Thus the equality $(\nu \circ \varphi)(f) = (\nu \circ \varphi)(g)$ means
  $$f(\nu(F_1), \ldots, \nu(F_n)) = g(\nu(F_1), \ldots, \nu(F_n)),$$
  in other words, 
  $$f(\nu(F_1), \ldots, \nu(F_n))(b) = g(\nu(F_1), \ldots, \nu(F_n))(b)$$
  for any $b \in B$.
  By Lemma \ref{elementwise}, this is equivalent to
  $$f(\nu(F_1)(b), \ldots, \nu(F_n)(b)) = g(\nu(F_1)(b), \ldots, \nu(F_n)(b))$$
  for any $b \in B$.
  Since $(f,g)$ is an arbitrary element in $\mathbf E(X_{\bm F})_{\mathbb B}$, this implies that
  $$\begin{pmatrix} \nu(F_1)(b) \\ \vdots \\ \nu(F_n)(b) \end{pmatrix} \in \mathbf V(\mathbf E(X_{\bm F})_{\mathbb B})$$
  for any $b \in B$.
\end{proof}

%Using Lemma \ref{key 1}, we give an observation.
%Let $R$ be a Laurent-generated subsemiring of $\ZAposu$, and let $\nu : R \to \ZBposu$ be a semiring homomorphism.
%Take a Laurent-generating set $\{ F_1, \ldots, F_n \}$ of $R$.

We want to show that $\nu$ is geometric.
That is, for any $b \in B$, we want to show that there exist some $a \in A$ and nonnegative rational number $t$ such that
$$\begin{pmatrix} \nu(F_1)(b) \\ \vdots \\ \nu(F_n)(b) \end{pmatrix} = t \begin{pmatrix} F_1(a) \\ \vdots \\ F_n(a) \end{pmatrix}.$$
This is equivalent to show that
$$\begin{pmatrix} \nu(F_1)(b) \\ \vdots \\ \nu(F_n)(b) \end{pmatrix} \in |X_{\bm F}|$$
for any $b \in B$.
Hence, to prove that our conjecture is true, it is sufficient to show that $\mathbf V(\mathbf E(X)_{\mathbb B}) = |X|$ for any (at most) 1-dimensional fan $X$ in $\mathbb R^n$.

\begin{lem}
  \label{result 1}
  If $\mathbf V(\mathbf E(X)_{\mathbb B}) = |X|$ for any 1-dimensional fan $X$ in $\mathbb R^n$ and for $X = \{\{ \mathbf 0 \}\}$, then Conjecture \ref{conj} is true.
\end{lem}

\subsection{Second key lemma and the proof of the main theorem}

Let $\rho$ be a ray in $\mathbb R^n$ and let $\bm d$ be a direction vector of it.
For $0 <\varepsilon < 1$, consider the following set:
$$U(\rho, \varepsilon) = \left\{ \bm p \in \mathbb R^n \setminus \{ \mathbf 0 \} \ \middle| \ \frac{\bm d \cdot \bm p}{\norm{\bm d} \norm{\bm p}} > \varepsilon \right\},$$
where $\cdot$ is the standard inner product and $\norm{\bm p} = \sqrt{\bm p \cdot \bm p}$.
It is easily checked that $U(\rho, \varepsilon)$ does not depend on the choice of $\bm d$.
Also it is easily checked that for any $\bm p \in \mathbb R^n \setminus \{ \mathbf 0 \}$,
$$\bm p \in U(\rho, \varepsilon) \Longleftrightarrow \mathbb R_{>0} \bm p \subset U(\rho, \varepsilon),$$
%We call a set of this form a \textit{open cone neighborhood} of $\rho$.
where $\mathbb R_{>0} \bm p = \{ t\bm p \ | \ t > 0 \}$.

For a subset $Z$ in $\mathbb R^n$, we consider the following condition $(*)$:
$$(*) \quad \text{$Z$ is nonempty closed subset of $\mathbb R^n$, and $\mathbb R_{\geq 0} Z = Z$},$$
where the topology is the Euclidean topology.
Note that both any 1-dimensional fan in $\mathbb R^n$ and $\{ \mathbf 0 \}$ satisfy $(*)$.
Also note that any set $Z$ satisfying $(*)$ includes $\mathbf 0$ because $Z$ is nonempty and $\mathbf 0 \in \mathbb R_{\geq 0}Z = Z$.

\begin{lem}
  \label{open cone nbd}
  Let $Z$ be a proper subset of $\mathbb R^n$ satisfying $(*)$.
  Take $\bm d \in \mathbb R^n \setminus Z$ and let $\rho$ be the ray spanned by $\bm d$.
  Then there exists $0 < \varepsilon < 1$ such that $U(\rho, \varepsilon) \cap Z = \emptyset$.
\end{lem}

\begin{proof}
  Since $Z$ is closed, there exists some $\delta > 0$ such that the open ball
  $$B(\bm d, \delta) := \{ \bm p \in \mathbb R^n \ | \ \norm{\bm p - \bm d} < \delta \}$$
  does not intersect to $Z$.
  We may assume that $\delta < \norm{\bm d}$.
  Let $\varepsilon = \frac{\sqrt{\norm{\bm d}^2 - \delta^2}}{\norm{\bm d}}$, which satisfies $0 < \varepsilon < 1$.
  Note that $\delta = \norm{\bm d}\sqrt{1- \varepsilon^2}$.
  We now show that
  $$U(\rho, \varepsilon) = \mathbb R_{>0} B(\bm d, \delta).$$

  Take $\bm p \in U(\rho, \varepsilon)$.
  The orthogonal projection of $\bm d$ onto $\mathbb R_{>0} \bm p$ is 
  $$\bm q := \frac{\bm d \cdot \bm p}{\norm{\bm p}^2} \bm p.$$
  Note that the coefficient is positive since if $\bm d \cdot \bm p \leq 0$, then $\bm p \not\in U(\rho, \varepsilon)$.
  The distance between $\bm d$ and $\bm q$ is estimated as
  $$\begin{aligned}
    \norm{\bm d - \bm q}^2 &= \norm{\bm d - \frac{\bm d \cdot \bm p}{||\bm p||^2} \bm p}^2\\
    &= \norm{\bm d}^2 - 2 \cdot \frac{(\bm d \cdot \bm p)^2}{\norm{\bm p}^2} + \frac{(\bm d \cdot \bm p)^2}{\norm{\bm p}^2}\\
    &= \norm{\bm d}^2 - \frac{(\bm d \cdot \bm p)^2}{\norm{\bm p}^2}\\
    &< \norm{\bm d}^2 - \frac {\left( \varepsilon \norm{\bm d}\norm{\bm p} \right)^2}{\norm{\bm p}^2}\\
    &= (1- \varepsilon^2)\norm{\bm d}^2\\
  %  &= \left( 1 - \frac{\norm{\bm d}^2 - \delta^2}{\norm{\bm d}^2} \right) \norm{\bm d}^2\\
    &= \delta^2.
  \end{aligned}$$
  Hence $\bm q \in B(\bm d, \delta)$, which means that $\bm p \in \mathbb R_{>0} B(\bm d, \delta)$.

  To show the converse, it is sufficient to show that $U(\rho, \varepsilon) \supset B(\bm d, \delta)$.
  Take $\bm p \in B(\bm d, \delta)$.
  Note that $\bm p \neq \mathbf 0$ since $\delta < \norm{\bm d}$.
  Then
%  $$\begin{aligned}
%    \frac {\bm d \cdot \bm p}{\norm{\bm d}\norm{\bm p}} &= \frac 12 \cdot \frac{\norm{\bm d}^2 + \norm{\bm p}^2 - \norm{\bm d - \bm p}^2}{\norm{\bm d}\norm{\bm p}}\\
%    &> \frac 12 \cdot \frac{\norm{\bm d}^2 + \norm{\bm p}^2 - \delta^2}{\norm{\bm d}\norm{\bm p}}\\
%    &= \frac 12 \cdot \frac{\norm{\bm d}^2 + \norm{\bm p}^2 - \norm{\bm d}^2(1-\varepsilon^2)}{\norm{\bm d}\norm{\bm p}}\\
%    &= \frac 12 \cdot \frac{\varepsilon^2 \norm{\bm d}^2 + \norm{\bm p}^2}{\norm{\bm d}\norm{\bm p}}\\
%    &= \frac 12 \left( \varepsilon^2 \frac {\norm{\bm d}}{\norm{\bm p}} + \frac {\norm{\bm p}}{\norm{\bm d}} \right)\\
%    &\geq \sqrt{\varepsilon^2 \frac {\norm{\bm d}}{\norm{\bm p}} \cdot \frac {\norm{\bm p}}{\norm{\bm d}}}\\
%    &= \varepsilon.
%  \end{aligned}$$
  $$\begin{aligned}
    {\bm d \cdot \bm p} &= \frac 12 (\norm{\bm d}^2 + \norm{\bm p}^2 - \norm{\bm d - \bm p}^2)\\
    &> \frac 12 (\norm{\bm d}^2 + \norm{\bm p}^2 - \delta^2)\\
    &= \frac 12 (\norm{\bm d}^2 + \norm{\bm p}^2 - \norm{\bm d}^2(1-\varepsilon^2))\\
    &= \frac 12 (\varepsilon^2 \norm{\bm d}^2 + \norm{\bm p}^2)\\
    &\geq \sqrt{\varepsilon^2 \norm{\bm d}^2  {\norm{\bm p}}^2}\\
    &= \varepsilon \norm{\bm d}\norm{\bm p}.
  \end{aligned}$$ 
  Hence $\bm p \in U(\rho, \varepsilon)$.
  Therefore $U(\rho, \varepsilon) = \mathbb R_{>0} B(\bm d, \delta)$.

  Assume that $U(\rho, \varepsilon) \cap Z \neq \emptyset$ and take $\bm p \in U(\rho, \varepsilon) \cap Z$.
  Then there exist some $t > 0$ such that $t \bm p \in B(\bm d, \delta)$.
  Since $\bm p \in Z$, $t\bm p$ is also in $Z$ by $(*)$, which contradicts to $B(\bm d, \delta) \cap Z = \emptyset$.
\end{proof}

The following is our second key lemma.

\begin{lem}
  \label{key 2}
  Let $\rho$ be a ray in $\mathbb R^n$ and let $0 < \varepsilon < 1$ be a real number.
  Then there exist two Boolean Laurent polynomial functions $f,g \in \Bxpmf$ such that
  \begin{enumerate}[(1)]
    \item $f|_{\rho} \neq g|_{\rho}$, and
    \item $f|_{\mathbb R^n \setminus U(\rho, \varepsilon)} = g|_{\mathbb R^n \setminus U(\rho, \varepsilon)}$.
  \end{enumerate}
\end{lem}

\begin{proof}
  Let $U = U(\rho, \varepsilon)$.
  Let $\bm d_1$ be a direction vector of $\rho$.
  Let $\bm d_2, \ldots, \bm d_n$ be a basis of the hyperplane orthogonal to $\bm d_1$.
  We claim that there exists $K>0$ such that for any $\bm p \in \mathbb R^n \setminus (U \cup \{ \mathbf 0 \})$,
  \begin{equation}\bm d_1 \cdot \bm p < K \max_{i=2, \ldots, n}  |\bm d_i \cdot \bm p|. \tag{$**$} \end{equation}
  If $\bm d_1 \cdot \bm p < 0$, $(**)$ is clear.
  Also $(**)$ holds if $\bm d_1 \cdot \bm p = 0$ because, otherwise, we have $\bm d_i \cdot \bm p = 0$ for $i=2, \ldots, n$, which means $\bm p= \mathbf 0$.
  This is a contradiction.
  Then we consider the vectors $\bm p$ such that $\bm d_1 \cdot \bm p > 0$.
  Let
  $$H = \{ \bm p \in \mathbb R^n \ | \ \bm d_1 \cdot \bm p > 0 \}.$$
  Consider the function
  $$h(\bm p) = \frac{\max_{i=2,\ldots, n} |\bm d_i \cdot \bm p|}{\bm d_1 \cdot \bm p}$$
  on $H \setminus U$.
  To show the claim, it is sufficient to show that the infimum of $h$ is positive.
  Indeed, if $\inf h = k > 0$, $h(\bm p) > \frac k2$ for any $\bm p \in H \setminus U$, and hence
  $$\bm d_1 \cdot \bm p_1 < \frac 2k \max_{i=2, \ldots, n}  |\bm d_i \cdot \bm p|.$$

  Since $h(t \bm p) = h(\bm p)$ for any $t > 0$ and $\bm p \in H \setminus U$, we may restrict the domain of $h$ to $S := S^{n-1} \cap (H \setminus U)$, where $S^{n-1}$ is the $(n-1)$-dimensional unit sphere in $\mathbb R^n$.
  Let
  $$H_0 := \{ \bm p \in \mathbb R^n \ | \ \bm d_1 \cdot \bm p = 0 \}.$$
  Then the closure $\overline S$ is
  $$\overline S = S \sqcup (S^{n-1} \cap H_0),$$
  which is compact.
  If $\bm p \in S$ tends to a point $\bm p_0$ in $S^{n-1} \cap H_0$, the denominator of $f(\bm p)$ converge to $0$.
  On the other hand, the numerator of $h(\bm p)$ does not converge to 0 because, otherwise,
  %then $h$ diverge to infinity because otherwise
  %$$\lim_{\bm p \to \bm p_0} \max_{i=2,\ldots,n} |\bm d_i \cdot \bm p| = 0.$$
  $\bm d_i \cdot \bm p_0 = 0$ for $i=2, \ldots, n$.
  Also $\bm d_1 \cdot \bm p_0 = 0$ since $\bm p_0 \in H_0$.
  Hence $\bm p_0=\mathbf 0$, which contradicts to $\bm p_0 \in S^{n-1}$.
  Therefore $\lim_{\bm p \to \bm p_0}h(\bm p) = \infty$.

  The value of $h$ is always nonnegative, hence $h$ is bounded from below.
  Since $\overline S$ is compact, $h$ attains its minimum at a point $\bm p_1$ in $S$.
  If $h(\bm p_1) = 0$, then $\bm d_i \cdot \bm p_1 = 0$ for $i=2, \ldots, n$.
  Also $\bm d_1 \cdot \bm p_1 > 0$ since $\bm p_1 \in H$.
  Hence $\bm p_1$ is a positive scalar multiple of $\bm d_1$, which contradicts to $\bm p_1 \not\in U$.
  Therefore $h(\bm p_1)>0$, which means that $\inf h >0$.

  Let $K>0$ be a number satisfying $(**)$ for any $\bm p \in \mathbb R^n \setminus (U \cup \{ \mathbf 0 \})$.
  If $\bm d_1, \ldots, \bm d_n$ are integer vectors and $K$ is an integer, then the functions
  $$f(\bm x) = \overline{\bm x^{K \bm d_2} \oplus \bm x^{-K \bm d_2} \oplus \cdots \oplus \bm x^{K \bm d_n} \oplus \bm x^{-K \bm d_n}}$$
  and
  $$g(\bm x) = \overline{\bm x^{\bm d_1}} \oplus f(\bm x)$$
  have the desired properties.
  However, we cannot always take $\bm d_i$'s as integer vectors.
  Thus we change the vectors as follows.

  Consider the following function on $(\mathbb R^n)^{n+1}$:
  $$\mathcal H(\bm p, \bm x_1, \ldots, \bm x_n) = \bm p \cdot \bm x_1 - \max_{i=2, \ldots, n} |\bm p \cdot \bm x_i|$$
  Then $\mathcal H$ is continuous.
  We know that
  $$ \mathcal H(\bm d_1, \bm d_1, K \bm d_2 \ldots, K \bm d_n) = \bm d_1 \cdot \bm d_1 >0.$$
  Hence there exists some $\varepsilon_1>0$ such that
  $$\norm{\bm x_1 - \bm d_1} < \varepsilon_1, \ \norm{\bm x_i - K \bm d_i}< \varepsilon_1 \ (i=2, \ldots, n) \Longrightarrow \mathcal H(\bm d_1, \bm x_1, \ldots, \bm x_n) >0.$$
  On the other hand, we know that, for any $\bm p \in \mathbb R^n \setminus (U \cup \{ \mathbf 0 \})$,
  \begin{equation*}\mathcal H(\bm p, \bm d_1, K \bm d_2, \ldots, K \bm d_n) < 0. \end{equation*}
  In particular, this inequality holds for any $\bm p \in S^{n-1} \setminus U$.
  Since $S^{n-1} \setminus U$ is compact, the function
  $$S^{n-1} \setminus U \to \mathbb R, \qquad \bm p \mapsto \mathcal H(\bm p, \bm d_1, K \bm d_2, \ldots, K \bm d_n)$$
  attains its maximum at some point in $S^{n-1} \setminus U$.
  This means that
  $$\sup_{\bm p \in S^{n-1} \setminus U} \mathcal H(\bm p, \bm d_1, K \bm d_2,  \ldots, K \bm d_n) < 0.$$
  Since $S^{n-1} \setminus U$ is compact, the function
  $$(\bm x_1, \ldots, \bm x_n) \mapsto \sup_{\bm p \in S^{n-1} \setminus U} \mathcal H(\bm p, \bm x_1, \ldots, \bm x_n)$$
  is continuous. %参考: https://williewong.wordpress.com/2011/11/01/continuity-of-the-infimum/
  Hence there exists some $\varepsilon_2 > 0$ such that
  \begin{multline*}
    \bm p \in S^{n-1} \setminus U, \ \norm{\bm x_1 - \bm d_1}< \varepsilon_2, \ \norm{\bm x_i - K \bm d_i}< \varepsilon_2 \ (i=2, \ldots, n) \\
    \Longrightarrow \mathcal H(\bm p, \bm x_1, \ldots, \bm x_n) < 0.
  \end{multline*}
  Since $\mathcal H(t \bm p, \bm x_1, \ldots, \bm x_n) = t \mathcal H(\bm p, \bm x_1, \ldots, \bm x_n)$ for any $t>0$ and $\bm p, \bm x_1, \ldots, \bm x_n \in \mathbb R^n$, we have
  \begin{multline*}
    \bm p \in \mathbb R^n \setminus (U \cup \{ \mathbf 0 \}), \ \norm{\bm x_1 - \bm d_1}< \varepsilon_2, \ \norm{\bm x_i - K \bm d_i}< \varepsilon_2 \ (i=2, \ldots, n) \\
    \Longrightarrow \mathcal H(\bm p, \bm x_1, \ldots, \bm x_n) < 0.
  \end{multline*}

  Take vectors $\bm e_1, \ldots, \bm e_n \in \mathbb Q^n$ such that
  $$\norm{\bm e_1 - \bm d_1}, \norm{\bm e_i - K \bm d_i} < \min \{ \varepsilon_1, \varepsilon_2 \} \ (i=2, \ldots, n).$$
  Then
  $$\mathcal H(\bm d_1, \bm e_1, \ldots, \bm e_n) >0$$
  and
  $$\bm p \in \mathbb R^n \setminus (U \cup \{ \mathbf 0 \}) \Longrightarrow \mathcal H(\bm p, \bm e_1, \ldots, \bm e_n) < 0.$$
  Take a positive integer $N$ so that $N\bm e_1, \ldots, N\bm e_n$ are integer vectors.
  Note that $\mathcal H(\bm p, N \bm e_1, \ldots, N \bm e_n) = N\mathcal H(\bm p, \bm e_1, \ldots, \bm e_n)$ for any $\bm p \in \mathbb R^n$.
  Hence we have
  $$\mathcal H(\bm d_1, N\bm e_1, \ldots, N\bm e_n) >0$$
  and
  $$\bm p \in \mathbb R^n \setminus (U \cup \{ \mathbf 0 \}) \Longrightarrow \mathcal H(\bm p, N\bm e_1, \ldots, N\bm e_n) < 0.$$
  Therefore the functions
  $$f(\bm x) = \overline{\bm x^{N \bm e_2} \oplus \bm x^{-N \bm e_2} \oplus \cdots \oplus \bm x^{N \bm e_n} \oplus \bm x^{-N \bm e_n}}$$
  and
  $$g(\bm x) = \overline{\bm x^{N \bm e_1}} \oplus f(\bm x)$$
  have the desired properties.
\end{proof}

\begin{cor}
  \label{result 2}
  For any subset $Z$ in $\mathbb R^n$ satisfying $(*)$, $\mathbf V(\mathbf E(Z)_{\mathbb B}) = Z$.
\end{cor}

\begin{proof}
  %If $Z = \mathbb R^n$, this is clear. 
  %Assume that $Z \subsetneq \mathbb R^n$.
  The inclusion $\supset$ is clear by definition.
  Conversely, take $\bm p \in \mathbf V(\mathbf E(Z)_{\mathbb B})$.
  Assume that $\bm p \not\in Z$.
  Then $Z$ is a proper subset of $\mathbb R^n$.
  Let $\rho$ be the ray spanned by $\bm p$.
  By Lemma \ref{open cone nbd}, there exists $0 < \varepsilon < 1$ such that $U(\rho, \varepsilon) \cap Z = \emptyset$.
  By Lemma \ref{key 2}, there exist two functions $f,g \in \Bxpmf$ such that
  \begin{enumerate}[(1)]
    \item $f|_{\rho} \neq g|_{\rho}$, and
    \item $f|_{\mathbb R^n \setminus U(\rho, \varepsilon)} = g|_{\mathbb R^n \setminus U(\rho, \varepsilon)}$.
  \end{enumerate}
  By (2) and $U(\rho, \varepsilon) \cap Z = \emptyset$, $f|_{Z} = g|_{Z}$.
  Hence $(f,g) \in \mathbf E(Z)_{\mathbb B}$.
  Since $\bm p \in \mathbf V(\mathbf E(Z)_{\mathbb B})$, $f(\bm p) = g(\bm p)$, and then $f|_{\rho} = g|_{\rho}$ by Lemma \ref{point ray}.
  It contradicts to (1).
\end{proof}

Now, the proof of our main theorem is completed.

\begin{thm}
  \label{main}
  Let $A,B$ be finite sets, let $R$ be a Laurent-generated subsemiring of $\ZAposu$, and let $\nu : R \to \ZBposu$ be a semiring homomorphism.
  Then $\nu$ is geometric.
\end{thm}

\begin{proof}
  It follows from Lemma \ref{result 1} and Corollary \ref{result 2}.
\end{proof}

%As we said in Section \ref{review ito2022local},

Also we have shown that the functor $\Phi$ (see Section \ref{review ito2022local}) is fully faithful.

\begin{cor}
  \label{ff}
  The functor $\Phi$ is fully faithful.
\end{cor}

\begin{proof}
  The fact that $\Phi$ is faithful is shown in \cite[Proposition 6.5]{ito2022local}.
  By Proposition \ref{full geom} and Theorem \ref{main}, $\Phi$ is full.
\end{proof}

%\begin{proof}
%  Let $\{ F_1, \ldots, F_n \}$ be a Laurent-generating set of $R$.
%  Then, by Lemma \ref{L-gen image}, $R$ is the image of the semiring homomorphism $\varphi : \Bxpm = \mathbb B[x_1^{\pm}, \ldots, x_n^{\pm}] \to R$, \ $\varphi(P) = P(F_1, \ldots, F_n)$.
%  For a pair $(P,Q)$ in $\Bxpm$, $(P,Q) \in \Ker \varphi$ if and only if $P(F_1, \ldots, F_n) = Q(F_1, \ldots, F_n)$, which means that
%  $$P(F_1, \ldots, F_n)(a) = Q(F_1, \ldots, F_n)(a)$$
%  for any $a \in A$.
%  By Lemma \ref{elementwise}, it is equivalent to
%  $$P(F_1(a), \ldots, F_n(a)) = Q(F_1(a), \ldots, F_n(a))$$
%  for any $a \in A$.
%  Now, for each $a \in A$, consider the following vector:
%  $$\bm p_a := \begin{pmatrix}
%    F_1(a) \\ \vdots \\ F_n(a)
%  \end{pmatrix}.$$
%  Thus $(P,Q) \in Ker \varphi$ if and only if $P(\bm p_a) = Q(\bm p_a)$ for any $a \in A$.
%\end{proof}

\section{Finding homomorphisms}

Let $A,B$ be finite sets.
As an application of our main theorem, we give a way to find all the semiring homomorphisms between given Laurent-generated subsemirings $R, R'$ of $\ZAposu$ and $\ZBposu$ respectively.
By Corollary \ref{ff}, this means that we can also find all the morphisms between 1-dimensional tropical fans.

We first see the case that $R' = \ZBposu$, which is easier than the other cases.
The result of this case will be used in the other cases.
%This is also a preparation for the other cases.

We need one more proposition.

\begin{prop}
  \label{exist}
  Let $R$ be a Laurent-generated subsemiring of $\ZAposu$, and let $\{ F_1, \ldots, F_n \}$ be a Laurent-generating set of $R$.
  Assume that the elements $G_1, \ldots, G_n \in \mathbb Z^B_0$ satisfy the following condition: for any $b \in B$, there exist some $a \in A$ and a nonnegative rational number $t$ such that
  $$\begin{pmatrix} G_1(b) \\ \vdots \\ G_n(b) \end{pmatrix} = t \begin{pmatrix} F_1(a) \\ \vdots \\ F_n(a) \end{pmatrix}.$$
  Then there exists a unique semiring homomorphism $\nu : R \to \ZBposu$ such that $\nu(F_i) = G_i$ for any $i$.
\end{prop}

\begin{proof}
  The uniqueness is clear since $R$ is Laurent-generated by $F_1, \ldots, F_n$.
  To show the existence, we use Lemma \ref{key 1} again.
  Let $\varphi : \Bxpmf \to \ZAposu$ be the semiring homomorphism defined as $\varphi(f) = f(F_1, \ldots, F_n)$.
  Then $\Im(\varphi) = R$ and $\Ker(\varphi) = \mathbf E(X_{\bm F})_{\mathbb B}$, where $\bm F = (F_1, \ldots, F_n)$.
  Hence $\varphi$ induces an isomorphism
  $$\overline {\varphi} : \Bxpmf / \mathbf E(X_{\bm F})_{\mathbb B} \to R.$$
  Also let $R'$ be the subsemiring of $\ZBposu$ Laurent-generated by $\{ G_1, \ldots, G_n \}$, and let $\psi : \Bxpmf \to \ZBposu$ be the semiring homomorphism defined as $\psi(f) = f(G_1, \ldots, G_n)$.
  %Then $\Im(\psi) = R'$ and $\Ker(\psi) = \mathbf E(X_{\bm G})_{\mathbb B}$, where $\bm G = (G_1, \ldots, G_n)$.
  Then $\psi$ induces an isomorphism
  $$\overline {\psi} : \Bxpmf / \mathbf E(X_{\bm G})_{\mathbb B} \to R',$$
  where $\bm G = (G_1, \ldots, G_n)$.

  By the assumption, we have $|X_{\bm G}| \subset |X_{\bm F}|$.
  This leads the inclusion $\mathbf E(X_{\bm G})_{\mathbb B} \supset \mathbf E(X_{\bm F})_{\mathbb B}$.
  Hence there is the canonical surjection
  $$\mu : \Bxpmf / \mathbf E(X_{\bm F})_{\mathbb B} \to \Bxpmf / \mathbf E(X_{\bm G})_{\mathbb B}.$$
  The composition
  $$R \xrightarrow{{\overline{\varphi}}^{-1}} \Bxpmf / \mathbf E(X_{\bm F})_{\mathbb B} \xrightarrow{\mu} \Bxpmf / \mathbf E(X_{\bm G})_{\mathbb B} \xrightarrow{\overline{\psi}} R' \hookrightarrow \ZBposu$$
  is the desired map.
\end{proof}

\begin{cor}
  \label{bij}
  Let $R$ be a Laurent-generated subsemiring of $\ZAposu$, and let $\{ F_1, \ldots, F_n \}$ be a Laurent-generating set of $R$.
  Consider the following two sets:
  $$\mathcal A := \{ \nu \ | \ \nu \text{ is a semiring homomorphism from $R$ to $\ZBposu$ } \},$$
  $$\mathcal B := \left\{ (G_1, \ldots, G_n) \in (\mathbb Z^B_0)^n  \ \middle| \ \begin{aligned}
    &\text{for any $b \in B$, there exist} \\
    &\text{some $a \in A$ and a nonnegative} \\
    &\text{rational number $t$ such that} \\
    &\text{$G_i(b) = tF_i(a)$ for any $i$}.
   \end{aligned} \right\}.$$
  Then the map
  $$\mathcal A \to \mathcal B, \qquad \nu \mapsto (\nu(F_1), \ldots, \nu(F_n))$$
  is bijective.
\end{cor}

\begin{proof}
  The map is well-defined by Theorem \ref{main}.
  The injectivity follows from the assumption that $\{ F_1, \ldots, F_n\}$ is a Laurent-generating set of $R$.
  The surjectivity follows from Proposition \ref{exist}.
\end{proof}

\begin{rem}
  \label{identify}
  If $A = \{ 1, \ldots, r \}$, an element of $\ZApos$ can be considered as a $r$-tuple $(a_1, \ldots, a_r)$ of integers such that $a_1+ \cdots +a_n \geq 0$.
  Let $\bm F=(F_1, \ldots, F_n)$ be an $n$-tuple of invertible elements in $\ZAposu$.
  The column vector
  $$\begin{pmatrix}
    F_1 \\ \vdots \\ F_n
  \end{pmatrix}$$
  can be identified with the $n \times r$ matrix
  $$M_{\bm F} := \begin{pmatrix}
    F_1(1) & \cdots & F_1(r) \\
    \vdots & & \vdots \\
    F_n(1) & \cdots & F_n(r)
  \end{pmatrix}.$$
  %whose sum of each row is $0$.
  
  Assume that $B = \{ 1, \ldots, s \}$.
  The set $\mathcal B$ in Corollary \ref{bij} is identified with the set
  $$\left\{ M \ \middle| \ \begin{aligned}
    &\text{ $M$ is an $n \times s$ integer matrix whose sum of each row is 0 and} \\
    &\text{ each column is a nonnegative scaler multiple of a column of $M_{\bm F}$ } \end{aligned} \right\}.$$
\end{rem}

\begin{ex}
  \label{ex1}
  Let $A = \{1,2,3,4,5\}$ and $B=\{1,2,3\}$.
  We use the identification in Remark \ref{identify}.

  Let
  $$F_1 := (1,-1,0,0,0), \quad F_2 := (0,0,1,-1,0), \quad F_3 := (1,1,1,1,-4),$$
  which are invertible elements in $\ZAposu$.
  Let $R$ be the subsemiring of $\ZAposu$ Laurent-generated by $\{ F_1, F_2, F_3 \}$.
  Let
  $$M_{\bm F} = \begin{pmatrix}
    1 & -1 & 0 & 0 & 0 \\
    0 & 0 & 1 & -1 & 0 \\
    1 & 1 & 1 & 1 & -4 \\
  \end{pmatrix}.$$
  
  Let us find all semiring homomorphisms from $R$ to $\ZBposu$.
  By Corollary \ref{bij} and Remark \ref{identify}, it is equivalent to find all $3 \times 3$ integer matrices such that
  \begin{enumerate}[(1)]
    \item the sum of each row is 0, and
    \item each column is a nonnegative scaler multiple of a column of $M_{\bm F}$.
  \end{enumerate}
  The zero matrix $O$ is one of such matrices.
  Let $M$ be another matrix.
  Since the sum of the third row of $M$ is 0, one of the columns of $M$ is a positive scalar multiple of $\begin{pmatrix}
    0 \\ 0 \\ -4
  \end{pmatrix}$, and another one is a positive scalar multiple of $\begin{pmatrix}
    1 \\ 0 \\ 1
  \end{pmatrix}, \begin{pmatrix}
    -1 \\ 0 \\ 1
  \end{pmatrix}, \begin{pmatrix}
    0 \\ 1 \\ 1
  \end{pmatrix},$ or $\begin{pmatrix}
    0 \\ -1 \\ 1
  \end{pmatrix}$.
  If one of the column of $M$ is a positive scalar multiple of $\begin{pmatrix} 1 \\ 0 \\ 1 \end{pmatrix}$, since the sum of the first row of $M$ is 0, one of the column of $M$ is a positive scalar multiple of $\begin{pmatrix} -1 \\ 0 \\ 1 \end{pmatrix}$.
  Hence $M$ is of the form.
  $$t\begin{pmatrix} 1 & -1 & 0 \\ 0 & 0 & 0 \\ 1 & 1 & -2 \end{pmatrix}, \qquad t \in \mathbb Z_{>0}$$
  or the matrix obtained by permutating the columns of it.
  The other cases are similar.
  Therefore, $M$ is one of the following forms:
  $$O, \qquad t\begin{pmatrix} 1 & -1 & 0 \\ 0 & 0 & 0 \\ 1 & 1 & -2 \end{pmatrix}, \qquad t\begin{pmatrix} 0 & 0 & 0 \\ 1 & -1 & 0 \\ 1 & 1 & -2 \end{pmatrix}, \qquad t \in \mathbb Z_{>0}$$
  or the matrices obtained by permutating the columns of them.
  %The set of semiring homomorphisms from $R$ to $\ZBposu$ consists of the homomorphisms corresponding to those matrices.
  For a matrix $M$ in the above list, the semiring homomorphism $\nu$ corresponding to $M$ is the unique one which satisfies that $\nu(F_i)$ is the $i$-th row of $M$ for $i=1,2,3$.
  For example, the semiring homomorphism $\nu$ corresponding to the matrix $\matcc 1{-1}000011{-2}$ is the unique one which satisfies $\nu(F_1) = (1, -1, 0), \ \nu(F_2) = (0, 0, 0)$, and $\nu(F_3) = (1, 1, -2)$.
%  It is interesting that the set of those matrices forms a 1-dimensional tropical fan in $\mathbb R^{3 \times 3}$.
\end{ex}

We now consider the case that $R'$ is an arbitrary Laurent-generated subsemiring of $\ZBposu$.
A semiring homomorphism $\nu : R \to R'$ is just a semiring homomorphism $\nu : R \to \ZBposu$ whose image is in $R'$.
Let $\{ F_1, \ldots, F_n \}$ be a Laurent-generating set of $R$.
Then $\Im(\nu) \subset R'$ if and only if $\nu(F_i) \in R'$ for any $i$.
Let $M$ be the matrix corresponding to $\nu$ via the identification in Remark \ref{identify}.
Recall that the rows of $M$ are $\nu(F_1), \ldots, \nu(F_n)$.
Hence the set of matrices corresponding to semiring homomorphisms from $R$ to $R'$ is
$$\left\{ M \ \middle| \ \begin{aligned}
  &\text{ $M$ is an $n \times |B|$ integer matrix whose sum of each row is 0,} \\
  &\text{ each column is a nonnegative scaler multiple of a column of $M_{\bm F}$,} \\
  &\text{ and each row is in $R'$}\end{aligned} \right\}.$$

\begin{ex}
  \label{ex3}
  We use the same notations to Example \ref{ex1}.
  Let
  $$G_1 := (1, -2, 1), \qquad G_2 := (1,2,-3),$$
  which are invertible elements in $\ZBposu$.
  Let $R'$ be the subsemiring of $\ZBposu$ Laurent-generated by $\{ G_1, G_2 \}$.

  We find all semiring homomorphisms from $R$ to $R'$.
  %Since $R'$ is a subsemiring of $\ZBposu$, it is also equivalent to finding all semiring homomorphisms from $R$ to $\ZBposu$ whose image is in $R'$.
  By the above argument, it is equivalent to finding all $3 \times 3$ matrices such that
  \begin{enumerate}[(1)]
    \item the sum of each row is 0,
    \item each row is a nonnegative scalar multiple of a column of $M_{\bm F}$, and
    \item each row is in $R'$.
  \end{enumerate}

  The condition \enquote{$(1)$ and $(2)$} is equivalent to that the matrix is one of the matrices found in Example \ref{ex1}.
  The condition \enquote{$(1)$ and $(3)$} is equivalent to that each row is in $R' \cap \mathbb Z^B_0$.  

  %Let $M$ be one of the matrices we found in Example \ref{ex1}, and let $\nu$ be the corresponding semiring homomorphism.
  %Recall that $R$ is Laurent-generated by $\{ F_1, F_2, F_3 \}$ and that $\nu(F_i)$ is the $i$-th row of $M$ for each $i$.
  %Hence $\Im(\nu)$ is Laurent-generated by the rows of $M$.
  %Therefore, $\Im(\nu)$ is in $R'$ if and only if the rows of $M$ are in $R'$. 
  %Recall that the sum of each row of $M$ is 0.
  %Thus each rows of $M$ is in $R'$ if and only if it is in $R' \cap \mathbb Z^B_0$.

  We claim that for any $a,b,c \in \mathbb Z$, $(a,b,c) \in R' \cap \mathbb Z^B_0$ if and only if
  \begin{enumerate}[(i)]
    \item $a+b+c=0$, and
    \item $a \equiv c \mod 4$.
  \end{enumerate}
  Indeed, by Lemma \ref{unit group}, $R' \cap \mathbb Z^B_0$ is the subgroup of $\mathbb Z^B_0$ generated by $\{ G_1,G_2 \}$.
  Since $G_1$ and $G_2$ satisfy (i) and (ii), any element of $R' \cap \mathbb Z^B_0$ satisfies (i) and (ii).
  Conversely, if $(a,b,c) \in \mathbb Z^3$ satisfies (i) and (ii),
  $$\begin{aligned}
    (a,b,c) &= (a,-a-c, c) \\
    &= a(1,-2,1) + (0, a-c, -a+c) \\
    &= a(1, -2, 1) + \frac {a-c}4 (0, 4, -4).
  \end{aligned}$$
  Note that $(0,4,-4) = (1, 2, -3) - (1, -2, 1) \in R' \cap \mathbb Z^B_0$.
  Hence $(a,b,c) \in R' \cap \mathbb Z^B_0$.

  Now, our purpose is finding matrices whose each row satisfies (i) and (ii) among the matrices found in Example \ref{ex1}.
%  Note that $(i)$ is always satisfied.
  By computing, we obtain the following list of matrices we are looking for:
  $$4t \matcc{1}{-1}{0}{0}{0}{0}{1}{1}{-2}, \qquad 4t \matcc{0}{0}{0}{1}{-1}{0}{1}{1}{-2}, \qquad t \in \mathbb Z_{>0},$$
  the matrices obtained by permutating the columns of them, 
  $$(4t-2) \matcc{1}{0}{-1}{0}{0}{0}{1}{-2}{1}, \qquad (4t-2) \matcc{0}{0}{0}{1}{0}{-1}{1}{-2}{1}, \qquad t \in \mathbb Z_{>0},$$
  the matrices obtained by exchanging the first and third columns of them, and $O$.
\end{ex}

Finally, we see a way to find all the morphisms between given 1-dimensional tropical fans.
Let $X, Y$ be 1-dimensional tropical fans in $\mathbb R^n$, $\mathbb R^m$ respectively.
Let $\varphi_X : \Bxpmf \to \ZXposu$ and $\varphi_Y : \Bypmf \to \ZYposu$ be the weighted evaluation maps of them respectively (see Section \ref{review ito2022local} for the definition).
By Corollary \ref{ff}, for any semiring homomorphism $\nu : \Im(\varphi_X) \to \Im(\varphi_Y)$, there exists a unique morphism $\mu : Y \to X$ such that $\Phi(\mu) = \nu$.
The way to construct $\mu$ was given in the proof of \cite[Proposition 6.8]{ito2022local}.
Explicitly, let $F_i := \varphi_X(\overline{x_i})$ and $G_j := \varphi_Y(\overline{y_j})$ for any $i,j$.
Then each $\nu(F_i)$ is in the group generated by $\{ G_1, \ldots, G_m \}$.
Hence there exist an $n \times m$ integer matrix $T$ such that
$$\matca {\nu(F_1)}{\vdots}{\nu(F_n)} = T \matca {G_1}{\vdots}{G_m}.$$
Thus the morphism $\mu$ is the restriction of the linear map $\mathbb R^m \to \mathbb R^n$ defined by $T$.

\begin{ex}
  We use the same notations to Example \ref{ex1} and Example \ref{ex3}.

  Let $\rho_1, \ldots, \rho_5$ be the rays in $\mathbb R^3$ spanned by
  $$\bm e_1=\matca 101, \ \bm e_2 = \matca {-1}01, \ \bm e_3 = \matca 011, \ \bm e_4 = \matca 0{-1}1, \ \bm e_5 = \matca 00{-1}$$
  respectively.
  Let $X = \{ \{ \mathbf 0 \}, \rho_1, \ldots, \rho_5 \}$.
  We give the weight function as
  $$\omega_X(\rho_1) = \omega_X(\rho_2) = \omega_X(\rho_3) = \omega_X(\rho_4) =1, \quad \omega_X(\rho_5) = 4.$$
  Then $X = (X, \omega_X)$ is a 1-dimensional tropical fan.
  Let $\varphi_X : \Bxpmf \to \ZXposu$ be the weighted evaluation map of $X$.
  Then
  $$\varphi_X(\overline{x_1}) = (1,-1,0,0,0) = F_1,$$
  $$\varphi_X(\overline{x_2}) = (0,0,1,-1,0) = F_2,$$
  $$\varphi_X(\overline{x_3}) = (1,1,1,1,-4) = F_3,$$
  where we identify $X(1) = \{\rho_1, \ldots, \rho_5\}$ with $A=\{1,2,3,4,5\}$.
  Hence $\Im(\varphi_X)$ coincides with the semiring $R$ in Example \ref{ex1}. 

  Also, let $\rho'_1, \rho'_2 , \rho'_3$ be the rays in $\mathbb R^2$ spanned by
  $$\bm e'_1=\matba 11, \ \bm e'_2 = \matba {-1}1, \ \bm e'_3 = \matba 1{-3}$$
  respectively.
  Let $Y = \{ \{ \mathbf 0 \}, \rho'_1, \rho'_2, \rho'_3 \}$.
  We give the weight function as
  $$\omega_Y(\rho'_1) = 1, \quad \omega_Y(\rho'_2) = 2, \quad \omega_Y(\rho'_3) = 1.$$
  Then $Y = (Y, \omega_Y)$ is a 1-dimensional tropical fan.
  Let $\varphi_Y : \mathbb B[\bm y^{\pm}]_{\mathrm{fcn}} \to \ZYposu$ be the weighted evaluation map of $Y$.
  Then
  $$\varphi_Y(\overline{y_1}) = (1, -2, 1) = G_1,$$
  $$\varphi_Y(\overline{y_2}) = (1,2,-3) = G_2,$$
  where we identify $Y(1) = \{\rho'_1, \rho'_2, \rho'_3\}$ with $B=\{1,2,3\}$.
  Hence $\Im(\varphi_Y)$ coincides with the semiring $R'$ in Example \ref{ex3}.

  We found all the semiring homomorphisms from $R$ to $R'$ in Example \ref{ex3}.
  This gives the complete list of morphisms from $Y$ to $X$.
  For example, we determine the morphism corresponding to the matrix
  $$M = 4\matcc 1{-1}000011{-2}.$$
  Let $\nu : R \to R'$ be the semiring homomorphism corresponding to $M$.
  Recall that $\nu(F_i)$ is the $i$-th row of $M$ for $i=1,2,3$.
  Thus we have
  $$\matca{\nu(F_1)}{\nu(F_2)}{\nu(F_3)} = \matcc 4{-4}000044{-8} = \matcb 310013 \matbc 1{-2}112{-3}.$$
  Then the corresponding morphism from $Y$ to $X$ is defined by the matrix $\matcb 310013$.
  By the similar computations, we obtain the following list of all matrices which define morphisms from $Y$ to $X$:
  $$t\matcb 310013, \quad t\matcb {-3}{-1}0013, \quad t \matcb {-1}100{-5}{-3}, \quad t\matcb 1{-1}00{-5}{-3},$$
  $$t \matcb 110020, \quad t \matcb {-1}{-1}0020$$
  for $t \in \mathbb Z_{>0}$, the matrices obtained by exchanging the first and second rows of them, and $O$.
\end{ex}

%\printbibliography

%\bibliography{bib}
%\bibliographystyle{plain}

%\begin{thebibliography}{99}
%  \bibitem{AR} First steps in tropical intersection theory
%  \bibitem{BE}
%  \bibitem{GG} The universal tropicalization and the Berkovich analytification
%  \bibitem{ito2022local}
%  \bibitem{JM}
%  \bibitem{S}
%\end{thebibliography}

%\printbibitembibliography

\end{document}